\documentclass[a4paper,12pt,reqno]{amsart}

\usepackage{amscd,amssymb,amsmath,amsfonts,amsthm, mathrsfs,xspace}
\usepackage{graphicx}
\usepackage{tensor}
\usepackage{verbatim,enumerate,paralist}
\usepackage{textcomp}
\usepackage[pagebackref,colorlinks]{hyperref}
\hypersetup{%
  urlcolor=blue,linkcolor=blue,citecolor=blue
}
\usepackage{pdfpages}

\DeclareMathAlphabet{\mathbbe}{U}{bbold}{m}{n}

\renewcommand{\epsilon}{\varepsilon}
\renewcommand{\phi}{\varphi}

\usepackage[T1]{fontenc}
\usepackage[all,2cell,poly,knot,tips]{xy}
\UseAllTwocells

\newcounter{Definitioncount}
\newtheorem{theorem}{Theorem}[section] 
\newtheorem{lemma}[theorem]{Lemma}
\newtheorem{proposition}[theorem]{Proposition}
\newtheorem{corollary}[theorem]{Corollary}

\theoremstyle{definition}
\newtheorem{remark}[theorem]{Remark}
\newtheorem{example}[theorem]{Example}

\newtheoremstyle{fact}{\bigskipamount}{\medskipamount}{\upshape}{}{\itshape}{. }{ }{Fact}
\theoremstyle{fact}

\newtheoremstyle{genquest}{\bigskipamount}{\medskipamount}{\upshape}{}{\itshape}{. }{ }{General Question}
\theoremstyle{genquest}

\newtheoremstyle{step}{2\bigskipamount}{\medskipamount}{\upshape}{}{\itshape}{. }{ }{\underline{Step~\thestep}}
\theoremstyle{step}

\renewcommand{\thestep}{\arabic{step}}

\newcommand{\dd}{\colon}
\newcommand{\lra}{\longrightarrow}

\newcommand{\ra}{\rightarrow}

\newcommand{\Ra}{\Rightarrow}

\newcommand{\ldual}[1]{\mathord{{\let\nolimits\relax\sideset{^\wedge}{}{#1}}}}
\newcommand{\laction}[2]{\mathord{{\let\nolimits\relax\sideset{^{#1}}{}{#2}}}}
\newcommand{\conj}[2]{\mathord{{\let\nolimits\relax\sideset{^{#1}}{}{#2}}}}

\def\CA{{\mathscr A}}
\def\CB{{\mathscr B}}
\def\CC{{\mathscr C}}

\def\CX{{\mathscr X}}

\renewcommand{\phi}{\varphi}
\renewcommand{\epsilon}{\varepsilon}

 \newcommand{\Set}{\ensuremath{\mathrm{Set}}\xspace}
 \newcommand{\Cat}{\ensuremath{\mathrm{Cat}}\xspace}

\newcommand{\ox}{\otimes}
\newcommand{\x}{\times}

\def\1c#1{\stackrel{#1}{\to}}

\begin{document}

\title{Skew-monoidal reflection and lifting theorems}

\author{Stephen Lack}
\address{Department of Mathematics, Macquarie University, NSW 2109, Australia}
\email{steve.lack@mq.edu.au}

\author{Ross Street}
\address{Department of Mathematics, Macquarie University, NSW 2109, Australia}
\email{ross.street@mq.edu.au}

\thanks{Both authors gratefully acknowledge the support of the Australian Research Council Discovery Grant DP130101969; Lack acknowledges with equal gratitude the support of an Australian Research Council Future Fellowship}

\dedicatory{In memory of Brian Day}
     
\date{\today}

\maketitle

\begin{abstract}
\noindent This paper extends the Day Reflection Theorem to skew monoidal categories.
We also provide conditions under which a skew monoidal structure can be lifted to
the category of Eilenberg-Moore algebras for a comonad.   

\end{abstract}

\tableofcontents

\section{Introduction}

In the first part of the paper, we squeeze some more results out of Brian Day's PhD thesis \cite{DayPhD}.
The question with which the thesis began was how to extend monoidal structures along 
dense functors, all at the level of enriched categories.
Brian separated the general problem into two special cases.
The first case concerned extending along a Yoneda embedding, 
which led to promonoidal categories
and Day convolution \cite{DayConv}. 
The second case involved extending along a reflection into a full subcategory:
the Day Reflection Theorem \cite{DayRT}. 

While the thesis was about monoidal categories, we can, without even modifying 
the biggest diagrams, adapt the results to skew monoidal categories.
Elsewhere \cite{115,116} we have discussed convolution.
Here we will provide the skew version of the Day Reflection Theorem \cite{DayRT}. 
The beauty of this variant is further evidence that the direction choices involved in the skew notion are important for organizing, and adding depth to, certain mathematical phenomena.  

In the second part of the present paper, the skew warpings of \cite{115} are 
slightly generalized to involve a skew action; they can in turn be seen as a special case of the skew warpings of \cite{121}.
Under certain natural conditions these warpings can be lifted to the category of 
Eilenberg-Moore coalgebras for a comonad. 
In particular, this applies to lift skew monoidal structures.
For idempotent comonads, we compare the result with our skew reflection theorem.  

\section{Skew monoidal reflection}

Recall from \cite{Szl2012,115,116} the notion of {\em (left) skew monoidal structure} on a category $\CX$. It involves a functor $\otimes : \CX \times \CX \lra \CX$, an object $I\in \CX$, and natural families of (not necessarily invertible) morphisms 
$$\alpha_{A,B,C} : (A\otimes B)\otimes C \ra A\otimes (B\otimes C) , 
\qquad \lambda_A : I\otimes A \ra A ,
\qquad \rho_A : I\ra A\otimes I ,$$
satisfying five coherence conditions. 
It was shown in \cite{JimA} that these five conditions are independent.

Recall, also from these references, that an {\em opmonoidal structure} 
on a functor $L\dd \CX \ra \CA$ consists of a natural family of morphisms
$$\psi_{X,Y} \dd L(X\otimes Y)\ra LX\bar{\otimes} LY$$
and a morphism $\psi_0 \dd LI \ra \bar{I}$ satisfying three axioms.
We say the opmonoidal functor is {\em normal} when $\psi_0$ is invertible.
We say the opmonoidal functor is {\em strong} when $\psi_0$ and all $\psi_{X,Y}$ are invertible. 
However, in this paper, a limited amount of such strength, 
in which only certain components of $\psi$ are invertible, will be important. 

Suppose $(\CX, \otimes, I, \alpha, \lambda,\rho)$ and 
$(\CA, \bar{\otimes}, \bar{I}, \bar{\alpha}, \bar{\lambda},\bar{\rho})$
are skew monoidal categories.

\begin{theorem}\label{skmonrefthm}
Suppose $L\dashv N \dd \CA\ra \CX$ is an adjunction with unit 
$\eta \dd 1_{\CX}\Ra NL$ and invertible counit $\varepsilon \dd LN \Ra 1_{\CA}$.
Suppose $\CX$ is skew monoidal.
There exists a skew monoidal structure on $\CA$ for which 
$L\dd \CX\ra \CA$ is normal opmonoidal with each 
$\psi_{X,NB}$ invertible if and only if, for all $X\in \CX$ and $B\in \CA$, the morphism 
\begin{eqnarray}\label{moninv}
L(\eta_X\otimes 1_{NB})\dd L(X\otimes NB)\ra L(NLX\otimes NB)
\end{eqnarray}
is invertible.
In that case, the skew monoidal structure on $\CA$ is unique up to isomorphism.  
\end{theorem}
\begin{proof}
Suppose $\CA$ has a skew monoidal structure $(\bar{\otimes}, \bar{I},\bar{\alpha}, \bar{\lambda},\bar{\rho})$ for which $L$ is normal opmonoidal with the $\psi_{X,NB}$ invertible.
We have the commutative square
$$
\xymatrix{
LX\bar{\otimes} LNB \ar[rr]^-{L\eta_X\bar{\otimes} 1} \ar[d]_-{\psi^{-1}} && LNLX\bar{\otimes} LNB \ar[d]^-{\psi^{-1}} \\
L(X\otimes NB) \ar[rr]^-{L(\eta_X\otimes 1)} && L(NLX\otimes NB)}
$$
in which the vertical arrows are invertible.
The top arrow is invertible with inverse $\varepsilon_{LX}\bar{\otimes}1$.
So the bottom arrow is invertible.

Conversely, suppose each $L(\eta_X\otimes 1_{NB})$ is invertible.
Wishing $L$ to become opmonoidal with the limited strength, we are forced 
(up to isomorphism) to put
$$A\bar{\otimes} B = L(NA\otimes NB) \ \text{ and } \ \bar{I} = LI \ ,$$
and to define the constraints $\bar{\alpha}, \bar{\lambda},\bar{\rho}$ by commutativity in the following diagrams. 
$$
\xymatrix{
L((NA\otimes NB)\otimes NC) \ar[rr]^-{L(\eta \otimes 1)} \ar[d]_-{L\alpha} && L(NL(NA\otimes NB)\otimes NC) \ar[d]^-{\bar{\alpha}} \\
L(NA\otimes (NA\otimes NC)) \ar[rr]_-{L(1\otimes \eta)} && L(NA\otimes NL(NB\otimes NC))}
$$
$$
\xymatrix{
L(I\otimes NA) \ar[rr]^-{L(\eta_I\otimes 1)} \ar[d]_-{L\lambda} && L(NLI\otimes NA) \ar[d]^-{\bar{\lambda}} \\
LNA \ar[rr]_-{\varepsilon_A} && A} 
\qquad
\xymatrix{
LNA \ar[rr]^-{\varepsilon_A} \ar[d]_-{L\rho} && A \ar[d]^-{\bar{\rho}} \\
L(NA\otimes I) \ar[rr]_-{L(1\otimes \eta_I)} && L(NA\otimes NLI)}
$$
The definitions make sense because the top arrows of the squares are invertible
(while the bottom arrows may not be). Now we need to verify the five axioms.
The proofs all proceed by preceding the desired diagram of barred morphisms
by suitable invertible morphisms involving only $\varepsilon_A$, $L\eta_X$, $\eta_{NA}$, or $L(\eta_X\otimes 1_{NB})$, then manipulating until one can make use of
the corresponding unbarred diagram. 

The biggest diagram for this is the proof of the pentagon for $\bar{\alpha}$.
Fortunately, the proof in Brian Day's thesis \cite{DayPhD} of the corresponding result for closed monoidal categories has the necessary Diagram 4.1.3 on page 94 written
without any inverse isomorphisms, so saves us rewriting it here. (The notation is a little different with $\psi$ in place of $N$ and with some of the simplifications we also use below.) 

It remains to verify the other four axioms. The simplest of these is 
\begin{eqnarray*}
\bar{\lambda}_{LI} \bar{\rho}_{LI} & = & \bar{\lambda}_{LI} \bar{\rho}_{LI}\varepsilon_{LI}L\eta_I \\
& = & \bar{\lambda}_{LI} L(1\otimes \eta_I) L\rho_{NLI}L\eta_I \\
& = & \bar{\lambda}_{LI} L(1\otimes \eta_I) L(\eta_I\otimes I) L\rho_I \\
& = & \bar{\lambda}_{LI}  L(\eta_I\otimes I) L(1\otimes \eta_I) L\rho_I \\
& = & \varepsilon_{LI} L\lambda_{NLI} L(1\otimes \eta_I) L\rho_I  \\
& = & \varepsilon_{LI} L\eta_I L\lambda_I L\rho_I  \\
& = &1_{LI}  L(\lambda_I \rho_I) \\
& = & 1_{LI} \ .
\end{eqnarray*}

For the other three, to simplify the notation (but to perhaps complicate the reading), 
we write as if $N$ were an inclusion of a full subcategory, choose $L$ 
so that the counit is an identity, and write $XY$ for $X\otimes Y$. 
Then we have
\begin{eqnarray*}
\bar{\lambda}_{B\bar{\otimes}C} \bar{\alpha}_{LI,B,C} L(\eta_{(LI)\bar{\ox}B}1_C)L((\eta_I1_B)1_C) & = & \bar{\lambda}_{B\bar{\otimes}C} L(1\eta_{BC})L\alpha_{LI,B,C} L((\eta_I1_B)1_C) \\
& = & \bar{\lambda}_{B\bar{\otimes}C} L(1_{LI}\eta_{BC})L(\eta_I1_{BC}) L\alpha_{I,B,C} \\
& = & \bar{\lambda}_{B\bar{\otimes}C} L(\eta_{I}1_{BC}) L(1_I\eta_{BC}) L\alpha_{I,B,C} \\
& = & L\lambda_{BC} L\alpha_{I,B,C} \\
& = & L(\lambda_{B} 1_C) \\
& = & (\bar{\lambda}_B\bar{\otimes}1_C)L(\eta_{(LI)B}1_C)L((\eta_I1_B)1_C) 
\end{eqnarray*}
yielding the axiom $\bar{\lambda}_{B\bar{\otimes}C} \bar{\alpha}_{LI,B,C} = \bar{\lambda}_B\bar{\otimes}1_C$ on right cancellation.

For the proof of the axiom 
$(1_A\bar{\otimes} \bar{\lambda}_C)\bar{\alpha}_{A,LI,C} (\bar{\rho}_A\bar{\otimes} 1_C) = 1_{A\bar{\otimes} C}$, 
we can look at Diagram 4.1.2 on page 93 of \cite{DayPhD}.
The required commutativities are all there once we reverse the direction of
the right unit constraint which Day calls $r$ instead of $\rho$. 

For the final axiom, we have
\begin{eqnarray*}
\bar{\alpha}_{A,B,LI} \bar{\rho}_{A\bar{\otimes}B} 
& = & \bar{\alpha}_{A,B,LI} L(\eta_{AB}1_{LI})L(1_{AB}\eta_I)L\rho_{AB} \\
& = & L(1_A\eta_{BLI}) L\alpha_{A,B,LI} L(1_{AB}\eta_I)L\rho_{AB} \\
& = & L(1_A\eta_{BLI}) L(1_A(1_B\eta_I)) L\alpha_{A,B,I} L\rho_{AB} \\
& = & L(1_A\eta_{BLI}) L(1_A(1_B\eta_I)) L(1_A\rho_B) \\
& = & 1_A\bar{\otimes} \bar{\rho}_B \ .  
\end{eqnarray*}

The desired opmonoidal structure on $L$ is defined by $\psi_0 = 1 \dd LI\ra \bar{I}$
and $\psi_{X,Y} = L(\eta_X\otimes \eta_Y) \dd L(X\otimes Y) \ra L(NLX\otimes NLY)$.
The three axioms for opmonoidality are easily checked and we have each $\psi_{X,NB} = L(1_{NLX}\otimes \eta_{NB})L(\eta_X\otimes 1_{NB})$ invertible. 
\end{proof} 

\section{A reflective lemma}\label{aLoR}

In this section we state a standard result in a form required for later reference.
For the sake of completeness, we include a proof.
 
Assume we have an adjunction $L\dashv N \dd \CA\ra \CX$ with unit $\eta \dd 1_{\CX}\Ra NL$
and counit $\varepsilon \dd LN \Ra 1_{\CA}$. 
Assume $N$ is fully faithful; that is, equivalently, the counit $\varepsilon$ is invertible.

\begin{lemma}\label{reflem}
For $Z\in \CX$, the following conditions are equivalent:
\begin{itemize}
\item[(i)] there exists $A\in \CA$ and $Z\cong NA$;
\item[(ii)] for all $X\in \CX$, the function $\CX(\eta_X,1)\dd \CX(NLX,Z)\ra \CX(X,Z)$ is surjective; 
\item[(iii)] the morphism $\eta_Z \dd Z\ra NLZ$ is a coretraction (split monomorphism);
\item[(iv)] the morphism $\eta_Z \dd Z\ra NLZ$ is invertible;
\item[(v)] for all $X\in \CX$, the function $\CX(\eta_X,1)\dd \CX(NLX,Z)\ra \CX(X,Z)$ is invertible.
\end{itemize}
\end{lemma} 
\begin{proof} $(i) \Rightarrow (ii)$ 
\begin{eqnarray*}
\xymatrix{
\CX(X,Z) \ar[r]^-{\cong} \ar[d]_-{1} & \CX(X,NA) \ar[r]^-{\cong} & \CA(LX,A) \ar[d]^-{N} \\
\CX(X,Z)  & \CX(NLX,Z) \ar[l]_-{\CX(\eta_X,1)} & \CX(NLX,NA)\ar[l]_-{\cong} }
\end{eqnarray*}

$(ii) \Rightarrow (iii)$ Take $X=Z$ and obtain $\nu \dd NLZ \ra Z$ with $\CX(\eta_Z,1)\nu=1_Z$.  

$(iii) \Rightarrow (iv)$ If $\nu \eta_Z=1$ then $(\eta_Z \nu) \eta_Z=1 \eta_Z$, so, by the universal property of $\eta_Z$, we have $\eta_Z \nu = 1$.

$(iv) \Rightarrow (v)$ The non-horizontal arrows in the commutative diagram
$$\xymatrix{
\CX(NLX,Z) \ar[d]_{\CX(1,\eta_Z)} \ar[rr]^{\CX(\eta_X,1)}  && \CX(X,Z) \ar[d]^{\CX(1,\eta_Z)} 
\\
\CX(NLX,NLZ)\ar[rr]^{\CX(\eta_X,1)} && \CX(X,NLZ) 
\\
& \CA(LX,LZ) \ar[lu]^-{N} \ar[ru]_-{\cong} 
}$$ 
are all invertible, so the horizontal arrows are invertible too.

$(v) \Rightarrow (i)$ Clearly $(v) \Rightarrow (ii)$ and we already have $(ii) \Rightarrow (iii)\Rightarrow (iv)$, so take $A=LZ$ and the invertible $\eta_Z$.
\end{proof} 

\section{Skew closed reflection}

The Reflection Theorem \cite{DayRT} also deals with closed structure.
 
If, for objects $Y$ and $Z$ the functor $\CX(-\ox Y,Z)$ is representable, 
say via a natural isomorphism 
$$\CX(X\ox Y,Z)\cong \CX(X,[Y,Z]),$$
we call the representing object $[Y,Z]$ a {\em left internal hom}. 
Recall from Section 8 of \cite{116} that if this exists for all $Z$, 
so that $-\ox Y$ has a right adjoint, then $\CX$ becomes left skew closed.

\begin{theorem}\label{skclosedrefthm}
Suppose $L\dashv N \dd \CA\ra \CX$ is an adjunction with unit 
$\eta \dd 1_{\CX}\Ra NL$ and invertible counit $\varepsilon \dd LN \Ra 1_{\CA}$.
Suppose $\CX$ is skew monoidal and left internal homs of the form $[NB,NC]$ exist for all $B,C\in \CA$.
The morphisms \eqref{moninv} are invertible for all $X \in \CX$ and $B\in \CA$ if and only if the morphisms 
\begin{eqnarray}\label{lclosedinv}
\eta_{[NB,NC]} \dd [NB,NC] \ra NL[NB,NC]
\end{eqnarray}
are invertible for all $B, C\in \CA$.
In that case, the skew monoidal structure abiding on $\CA$, as seen from Theorem~\ref{skmonrefthm}, is left closed. Also, the functor $N$ is strong left closed. 
\end{theorem}
\begin{proof}
Consider the following commutative diagram.
$$
\xymatrix{
\CA(L(NLX\otimes NB),C) \ar[rr]^-{\CA(L(\eta \otimes 1),1)} \ar[d]_-{\cong} && \CA(L(X\otimes NB),C) \ar[d]^-{\cong} \\
\CX(NLX\otimes NB,NC)  \ar[d]_-{\cong} && \CX(X\otimes NB,NC) \ar[d]^-{\cong} \\
\CX(NLX,[NB,NC]) \ar[rr]_-{\CX(\eta_X,1)} && \CX(X,[NB,NC])}
$$
Invertibility of the arrows \eqref{moninv} is equivalent to the invertibility of the top horizontal
arrows. 
This is equivalent to invertibility of the bottom horizontal arrows.
By Lemma~\ref{reflem}, this is equivalent to invertibility of the arrows \eqref{lclosedinv}.

For the penultimate sentence of the Theorem, we now have the natural isomorphisms:
 \begin{eqnarray*}
\CA(A\bar{\otimes} B,C) & \cong & \CX(NA\otimes NB,NC) \\
& \cong & \CX(NA,[NB,NC]) \\
& \cong & \CX(NA,NL[NB,NC]) \\
& \cong & \CA(A,L[NB,NC])  
\end{eqnarray*}
yielding the left internal hom $[B,C]=L[NB,NC]$ for $\CA$.
For the last sentence, we have $N[B,C]=NL[NB,NC]\cong [NB,NC]$. 
\end{proof}

Our notation for a right adjoint to $X\otimes -$ is
$$\CX(X\otimes Y,Z) \cong \CX(Y,\langle X,Z\rangle) \ .$$
The {\em right internal hom} $\langle X,Z\rangle$ may exist for only certain objects $Z$.
In general, the existence of right homs in a left skew monoidal category does not
give a left or right skew closed structure.
When they do exist, we can reinterpret a stronger form of the invertibility condition \eqref{moninv}
of Theorem~\ref{skmonrefthm}.
\begin{theorem}\label{rskclosedrefthm}
Suppose $L\dashv N \dd \CA\ra \CX$ is an adjunction with unit 
$\eta \dd 1_{\CX}\Ra NL$ and invertible counit $\varepsilon \dd LN \Ra 1_{\CA}$.
Suppose $\CX$ is skew monoidal, and left internal homs of the form $[Y,NC]$ 
and right internal homs of the form $\langle X,NC\rangle$ exist.
The invertibility of one of the following three natural transformations implies invertibility of the other two: 
\begin{eqnarray}\label{moninvY}
L(\eta_X\otimes 1_{Y})\dd L(X\otimes Y)\ra L(NLX\otimes Y) \ ;
\end{eqnarray}
\begin{eqnarray}\label{lclosedinvY}
\eta_{[Y,NC]} \dd [Y,NC] \ra NL[Y,NC] \ ;
\end{eqnarray}
\begin{eqnarray}\label{rclosedinvY}
{\langle \eta_X,NC\rangle} \dd \langle NLX,NC\rangle \ra \langle X,NC\rangle \ .
\end{eqnarray}
\end{theorem}
\begin{proof}
Consider the commutative diagram \eqref{rightproof}.
Invertibility of any one of the horizontal families in the diagram implies
that of the other two.
Invertibility of the arrows \eqref{moninvY} is equivalent to the invertibility of the top horizontal family. 
By Lemma~\ref{reflem}, invertibility of the middle horizontal family is equivalent to invertibility of the arrows \eqref{lclosedinv}.
By the Yoneda Lemma, invertibility of the bottom horizontal family is equivalent to invertibility of the arrows \eqref{rclosedinvY}.
\end{proof}

\begin{eqnarray}\label{rightproof}
\begin{aligned}
\xymatrix{
\CA(L(NLX\otimes Y),C) \ar[rr]^-{\CA(L(\eta \otimes 1),1)} \ar[d]_-{\cong} && \CA(L(X\otimes Y),C) \ar[d]^-{\cong} \\
\CX(NLX\otimes Y,NC)  \ar[d]_-{\cong} && \CX(X\otimes Y,NC) \ar[d]^-{\cong} \\
\CX(NLX,[Y,NC]) \ar[rr]^-{\CX(\eta_X,1)} \ar[d]_-{\cong} && \CX(X,[Y,NC]) \ar[d]_-{\cong} \\
\CX(Y,\langle NLX,NC\rangle) \ar[rr]^-{\CX(1,\langle \eta_X,1\rangle)} && \CX(Y,\langle X,NC\rangle)}
\end{aligned}
\end{eqnarray}

\section{An example}\label{ex}

This is an example of the opposite (dual) of Theorem~\ref{skmonrefthm}
which we enunciate explicitly as Proposition~\ref{skmoncorefprop} below.
Instead of a reflection we have a coreflection.
To keep using left skew monoidal categories we also reverse the tensor product.
For a monoidal functor $R \dd \CX \to \CA$, we denote the structural morphisms
by $$\phi_0\dd I \to RI \ \text{ and } \ \phi_{X,Y}\dd RX\otimes RY\to R(X\ox Y) \ .$$ 

\begin{proposition}\label{skmoncorefprop}
Suppose $R\vdash N \dd \CA\ra \CX$ is an adjunction with counit 
$\varepsilon \dd NR \Ra 1_{\CX}$ and invertible unit $\eta \dd 1_{\CA}\Ra RN$.
Suppose $\CX$ is left skew monoidal.
There exists a left skew monoidal structure on $\CA$ for which 
$R\dd \CX\ra \CA$ is normal monoidal each $\phi_{NA,Y}$ invertible
if and only if, for all $A\in \CA$ and $Y\in \CX$, the morphism 
\begin{eqnarray}\label{moninv}
R(NA\otimes \varepsilon_Y)\dd R(NA\otimes NRY)\ra R(NA\otimes Y)
\end{eqnarray}
is invertible. 
\end{proposition}

Consider an injective function $\mu \dd U\ra O$.
For an object $A$ of the slice category $\mathrm{Set}/U$, we write
$A_u$ for the fibre over $u\in U$. We have an adjunction 
$$R\vdash N \dd \mathrm{Set}/U \ra \mathrm{Set}/O$$ defined by
$(NA)_i = \sum_{\mu(u)=i}{A_u}$ and $(RX)_u = X_{\mu(u)}$ with invertible unit.
The $i$th component of the counit $\varepsilon_X\dd NRX \ra X$ is the function 
$\sum_{\mu(u)=i}{X_{\mu(u)} \ra X_i}$ which is the identity of $X_i$ when $i$ 
is in the image of $\mu$. 

Let $\CC$ be a category with $\mathrm{ob}\CC = O$. 
Then $\mathrm{Set}/O$ becomes left skew monoidal on defining the tensor $X\otimes Y$ by
$$(X\otimes Y)_j = \sum_i{X_i\times \CC(i,j)\times Y_j}$$
and the (skew) unit $I$ by $I_j=1$.
The associativity constraint $\alpha \dd (X\otimes Y)\otimes Z \ra X\otimes (Y\otimes Z)$
is defined by the component functions
$$\sum_{i,j}{X_i\times \CC(i,j)\times Y_j\times \CC(j,k)\times Z_k}\ra \sum_{i,j}{X_i\times \CC(i,k)\times Y_j\times \CC(j,k)\times Z_k}$$ 
induced by the functions
$$ \CC(i,j)\times \CC(j,k) \ra  \CC(i,k)\times \CC(j,k)$$
taking $(a \dd i\ra j,b\dd j\ra k)$ to $(b\circ a \dd i\ra k, b\dd j\ra k)$. 
Define $\lambda_Y\dd I\otimes Y\ra Y$ to have $j$-component $\sum_i{\CC(i,j)\times Y_j} \ra Y_j$
whose restriction to the $i$th injection is the second projection onto $Y_j$.
Define $\rho_X \dd X\ra X\otimes I$ to have $j$-component $X_j \ra \sum_i{X_i\times \CC(i,j)}$
equal to the composite of $X_j\ra X_j\times \CC(j,j), \ x \mapsto (x,1_j),$ with the $j$th injection.

This provides an example of Proposition~\ref{skmoncorefprop}.
In fact, it satisfies the stronger condition of the dual to Theorem~\ref{rskclosedrefthm}. 
To see that $$R(X \otimes \varepsilon_Y)\dd R(X\otimes NRY)\ra R(X\otimes Y)$$
is invertible, since $N$ is fully faithful, we need to prove
$$G(X \otimes \varepsilon_Y)\dd G(X\otimes GY)\ra G(X\otimes Y)$$
is invertible where $G = NR$ is the idempotent comonad generated by the reflection.
Notice that $(GX) = U_j\times X_j$ where $U_j$ is the fibre of $\mu$ over $j\in O$.
Since $\mu$ is injective, $U_j\cong U_j\otimes U_j$, so   
\begin{eqnarray*}
G(X\otimes GY)_j & = &U_j\times (X\otimes GY)_{j} \\
& = & U_j\times \sum_i{X_i\times \CC(i,j)\times (GY)_j} \\
& = & \sum_i{U_j\times X_i\times \CC(i,j)\times U_j\times Y_j} \\
& \cong & \sum_i{U_j\times X_i\times \CC(i,j) \times Y_j} \\
& = & U_j\times (X\otimes Y)_{j} \\
& = & G(X\otimes Y)_{j} \ .  
\end{eqnarray*}
The resultant left skew structure on $\mathrm{Set}/U$ has tensor product
\begin{eqnarray*}
(A\bar{\otimes} B)_v & = & R(NA\otimes NB)_v \\
& = & (NA\otimes NB)_{\mu(v)} \\
& = & \sum_i{(NA)_i\times \CC(i,\mu(v))\times (NB)_{\mu(v)}} \\
& \cong & \sum_u{A_{u}\times \CC(\mu(u),\mu(v))\times B_v} \ .  
\end{eqnarray*}
Of course we can see that this is merely the left skew structure on $\mathrm{Set}/U$
arising from the category whose objects are the elements $u\in U$ and whose
morphisms $u\ra v$ are morphisms $\mu(u)\ra \mu(v)$ in $\CC$; 
that is, the category arising as the full image of the functor $\mu \dd U\ra \CC$.

As an easy exercise the reader might like to calculate the monoidal structure
$$RX\bar{\otimes} RY \ra R(X\otimes Y)$$ on $R$ and check that these components
are not invertible in general while, of course, they are for $X=NA$.

\section{Skew warpings riding a skew action}\label{sw}

We slightly generalize the notion of skew warping defined in \cite{115} to involve an action.
This is actually a special case of skew warping on a two-object skew bicategory in
the sense of \cite{121}.  

Let $\CC$ denote a left skew monoidal category. 
A {\em left skew action} of $\CC$ on a category $\CA$ is an opmonoidal
functor 
\begin{eqnarray}\label{lsa}
\CC \lra [\CA,\CA] \ , \ X \mapsto X\star -
\end{eqnarray}
where the skew monoidal (in fact strict monoidal) tensor product on the endofunctor category $[\CA,\CA]$
is composition.
The opmonoidal structure on \eqref{lsa} consists of natural families
\begin{eqnarray}\label{lsas}
\alpha_{X,Y,A} \dd (X\ox Y)\star A \lra X\star (Y\star A) \ \text{ and } \ \lambda_A \dd I\star A \lra A
\end{eqnarray}
subject to the three axioms \eqref{lsa1}, \eqref{lsa2}, \eqref{lsa3}.
\begin{eqnarray}\label{lsa1}
\begin{aligned}
\xymatrix{
((X\ox Y)\ox Z)\star A \ar[rr]^-{\alpha} \ar[d]_-{\alpha\star 1} &  & (X\ox Y)\star (Z\star A) \ar[d]^-{\alpha} \\
(X\ox (Y\ox Z))\star A \ar[r]_-{\alpha} & X\star ((Y\ox Z)\star A) \ar[r]_-{1\star\alpha} & X\star (Y\star (Z\star A)) }
\end{aligned}
\end{eqnarray}
\begin{eqnarray}\label{lsa2}
\begin{aligned}
\xymatrix{
(I\ox Y)\star A \ar[rd]_{\lambda\star 1}\ar[rr]^{\alpha}   && I\star (Y\star A) \ar[ld]^{\lambda} \\
& Y\star A  &
}
\end{aligned}
\end{eqnarray}
\begin{eqnarray}\label{lsa3}
\begin{aligned}
\xymatrix{
(X\ox I)\star A \ar[rr]^-{\alpha}  && X\star (I\star A) \ar[d]^-{1\star \lambda} \\
X\star A \ar[rr]_-{1} \ar[u]^-{\rho\star 1} && X\star A}
\end{aligned}
\end{eqnarray}
A category $\CA$ equipped with a skew action of $\CC$ is called a {\em skew $\CC$-actegory}.

A {\em skew left warping} riding the skew action of $\CC$ on $\CA$ consists of the following data:
\begin{itemize}
\item[(a)] a functor $T:\CA \lra \CC$;
\item[(b)] an object $K$ of $\CA$; 
\item[(c)] a natural family of morphisms $v_{A,B} : T(TA\star B)\lra TA\otimes TB$ in $\CC$;
\item[(d)] a morphism $v_0 : TK \lra I$; and,
\item[(e)] a natural family of morphisms $k_A : A \lra TA\star K$;
\end{itemize}
such that the following five diagrams commute.
\begin{equation}\label{warpassoc}
\begin{aligned}
\xymatrix{
T(TA\star B)\otimes TC \ar[rr]^-{v_{A,B}\otimes 1} && (TA\otimes TB)\otimes TC \ar[d]^-{\alpha_{TA,TB,TC}} \\
T(T(TA\star B)\star C) \ar[u]^-{v_{TA\star B , C}} \ar[d]_-{T(v_{A,B}\star 1)} && TA\otimes (TB\otimes TC) \\
T((TA\otimes TB)\star C) \ar[dr]_-{T\alpha_{TA,TB,C}\phantom{AA}} && TA\otimes T(TB\star C) \ar[u]_-{1\otimes v_{B,C}} \\
& T(TA\star (TB\star C)) \ar[ru]_-{\phantom{AA}v_{A,TB\star C}} &}
\end{aligned}
\end{equation}
\begin{equation}\label{warpunit1}
\begin{aligned}
\xymatrix{
& TK\otimes TB \ar[rd]^-{ v_0 \otimes 1_{TB}}  & \\
T(TK\star B) \ar[ru]^-{v_{K,B}} \ar[d]_-{T(v_0\star 1_B)} & & I\otimes TB \ar[d]^-{\lambda_{TB}} \\
T(I\otimes B) \ar[rr]_-{T\lambda_B} & & TB }
\end{aligned}
\end{equation}
\begin{equation}\label{warpunit2}
\begin{aligned}
\xymatrix{
T(TA\star K)  \ar[rr]^-{v_{A,K}} && TA\otimes TK \ar[d]^-{1\otimes v_{0}} \\
TA \ar[u]^-{Tk_{A}}  \ar[rr]_{\rho_{TA}} && TA\ox I 
}\end{aligned}
\end{equation}
\begin{equation}\label{warpunit3}
\begin{aligned}
\xymatrix{
T(TA\star B)\star K  \ar[rr]^-{v_{A,B}\star 1_K} && (TA\otimes TB)\star K \ar[d]^-{\alpha_{TA,TB,K}} \\
TA\star B \ar[u]^-{k_{TA\star B}}  \ar[rr]_{1_{TA}\star k_B} && TA\star (TB\star K) }
\end{aligned}
\end{equation}
\begin{equation}\label{warpunit4}
\begin{aligned}
\xymatrix{
TK\star K \ar[rr]^-{v_0\star1_K}  && I\star K \ar[d]^-{\lambda_K} \\
K \ar[rr]_-{1_K} \ar[u]^-{k_K} && K}
\end{aligned}
\end{equation}

\begin{example}
A skew warping on a skew monoidal category (in the sense of \cite{115}) is just the case where 
$\CA=\CC$ with tensor as action.
\end{example}

Just as in Proposition 3.6 of \cite{115}, we obtain a skew monoidal structure from a skew warping.

\begin{proposition}\label{newskew}
A skew left warping riding a left skew action of a left skew monoidal category $\CC$ on a category $\CA$ determines left skew monoidal structure on $\CA$ as follows:
\begin{itemize}
\item[(a)] tensor product functor $A\bar{\ox} B = TA\star B$;
\item[(b)] unit $K$;
\item[(c)] associativity constraint $$T(TA\star B)\star C \stackrel{v_{A,B}\star 1_C}\lra (TA\otimes TB) \star C \stackrel{\alpha_{TA,TB,C}}\lra TA\star (TB\star C) \ ;$$
\item[(d)] left unit constraint $$TK\star B\stackrel{v_0\star 1_B}\lra I\star B \stackrel{\lambda_B}\lra B \ ;$$
\item[(e)] right unit constraint $$A\stackrel{k_A} \lra TA\star K \ .$$
\end{itemize}
There is an opmonoidal functor $(T, v_0 , v_{A,B}) : (\CA , \bar{\ox} , K) \lra (\CC , \otimes , I)$. 
\end{proposition}

\begin{example}
Skew warpings are more basic than skew monoidal structures in the following sense. 
Just pretend, for the moment, that we do not know what a skew monoidal (or even monoidal)
category is, except that we would like endofunctor categories to be examples. 
For any category $\CA$, the endofunctor category $\CC = [\CA,\CA]$ acts on $\CA$ by evaluation;
as a functor \eqref{lsa}, the action is the identity. 
A left skew warping riding this action could be taken as the definition of a left skew monoidal structure on $\CA$. 
\end{example}

\section{Comonads on skew actegories}

For a left skew monoidal category $\CC$, let $\Cat^\CC$ denote the 2-category whose objects
are left skew $\CC$-actegories as defined in Section~\ref{sw}.
A morphism is a functor $F \dd \CA \to \CB$ equipped with a natural family of morphisms
\begin{eqnarray}
\gamma_{X,A}\dd X\star FA \lra F(X\star A)
\end{eqnarray}
such that \eqref{am1} and \eqref{am2} commute. 
\begin{equation}\label{am1}
\begin{aligned}
  \xymatrix{
(X\ox Y)\star FA \ar[rr]^{\gamma} \ar[d]_{\alpha} && F((X\ox Y)\star A) \ar[d]^{F\alpha} \\
X\star(Y\star FA) \ar[r]_{1\star\gamma} & X\star F(Y\star A) \ar[r]_{\gamma} &F( X\star(Y\star A)) }
\end{aligned}
\end{equation}
\begin{equation}
  \label{am2}
  \begin{aligned}
  \xymatrix{
I\star FA \ar[r]^{\gamma} \ar[dr]_{\lambda} & F(I\star A) \ar[d]^{F\lambda} \\ & FA }
\end{aligned}
\end{equation}
Such a morphism is called {\em strong} when each $\gamma_{X,A}$ is invertible.
A 2-cell $\xi \dd (F,\gamma) \Rightarrow (G,\gamma)$ in $\Cat^\CC$ is a natural transformation $\xi \dd F \Rightarrow G$ such that \eqref{am2cell} commutes.
\begin{equation}
  \label{am2cell}
  \begin{aligned}
  \xymatrix{
X\star FA \ar[r]^{\gamma_{X,A}} \ar[d]_{1\star\xi_A} & F(X\star A) \ar[d]^{\xi_{X\star A}} \\
X\star GA \ar[r]_{\gamma_{X,A}} & G(X\star A) }
\end{aligned}
\end{equation}

As usual with actions, there is another way to view the 2-category $\Cat^\CC$.
Regard $\CC$ as the homcategory of a 1-object skew bicategory $\Sigma\CC$ in the sense of
Section 3 of \cite{121}. A left skew $\CC$-actegory is an oplax functor $\CA \dd \Sigma\CC \to \Cat$.
A morphism $(F,\gamma) \dd \CA \to \CB$ in $\Cat^\CC$ can be identified with a lax natural transformation
between the oplax functors. The 2-cells are the modifications.  

We are interested in comonads $(\CA, G,\gamma, \delta, \epsilon)$ in the 2-category $\Cat^\CC$.
These are objects of the 2-category $\mathrm{Mnd}_*(\Cat^\CC)$ as defined in \cite{3}.   
Alternatively, they are oplax functors $\Sigma\CC \to \mathrm{Mnd}_*(\Cat)$.
For later reference, apart from the conditions for being a comonad on $\CA$ and the conditions
\eqref{am1} and \eqref{am2}, we require commutativity of \eqref{am3}. 
\begin{equation}  \label{am3}
\begin{aligned}
\xymatrix{
X\star GA \ar[rr]^{\gamma} \ar[d]_{1\star \delta} & & G(X\star A) \ar[d]^{\delta} \\
X\star G^2A \ar[r]_{\gamma} & G(X\star GA) \ar[r]_{G\gamma} & G^2(X\star A) } \quad
\xymatrix{
X\star GA \ar[r]^{\gamma} \ar[dr]_{1\star\epsilon} & G(X\star A) \ar[d]^{\epsilon} \\
& X\star A } 
\end{aligned} 
\end{equation}

The Eilenberg-Moore coalgebra construction $(\CA,G, \delta, \epsilon)\mapsto \CA^G$ is the 2-functor 
right adjoint to the 2-functor $\Cat \to \mathrm{Mnd}_*(\Cat)$ taking each category to that 
category equipped with its identity comonad. 

\begin{proposition}\label{coalgskewaction}
For each comonad $(\CA, G,\gamma, \delta, \epsilon)$ in the 2-category $\Cat^\CC$, the Eilenberg-Moore
coalgebra category $\CA^G$ becomes a left skew $\CC$-actegory with skew action 
\begin{eqnarray*}
X\star (A \stackrel{a}\lra GA) = (X\star A, X\star A \stackrel{X\star a \ }\lra X\star GA \stackrel{\gamma_{X,A} \ }\lra G(X\star A)) \ .  
\end{eqnarray*}
This provides the Eilenberg-Moore construction in the 2-category $\Cat^\CC$ (in the sense of \cite{3}). 
\end{proposition}
\begin{proof}
Compose the oplax functor $\Sigma\CC \to \mathrm{Mnd}_*(\Cat)$ corresponding to $(\CA, G,\gamma, \delta, \epsilon)$ with the Eilenberg-Moore 2-functor $\mathrm{Mnd}_*(\Cat) \to \Cat$.  
\end{proof}

Let $\mathrm{U}\dd \CA^G \to \CA$ denote the underlying functor $(A,a)\mapsto A$.

\begin{proposition}\label{liftwarping}
Suppose $(T,K,v,v_0,k)$ is a skew left warping riding the $\CC$-actegory $\CA$.
Suppose $(\CA, G,\gamma , \delta, \epsilon)$ is a comonad in the 2-category $\Cat^\CC$
for which all morphisms of the form $\gamma_{TA,K}$ and $\gamma_{TA,TB\star K}$ are invertible.
Then $(T\mathrm{U},(GK,\delta_K),v,v'_0,k')$ is a skew left warping riding the $\CC$-actegory $\CA^G$ of 
Proposition~\ref{coalgskewaction}, where $v'_0=v_0\circ T\epsilon_K$ and 
$k'_{(A,a)} = \gamma^{-1}_{TA,K}\circ Gk_A\circ a$. 
\end{proposition}
\begin{proof}
First we need to see that $k'_{(A,a)}\dd (A,a) \to (TA\star GK, \gamma_{TA,GK}\circ (1\star \delta_K))$
is a $G$-coalgebra morphism. This uses the first diagram of \eqref{am3}, naturality of $\delta$
with respect to $k_A$, and the coassociativity of the coaction $a\dd A\to GA$.  

It remains to verify the five axioms \eqref{warpassoc}, \eqref{warpunit1}, \eqref{warpunit2}, \eqref{warpunit3}, \eqref{warpunit4}.
Since only $v$ is involved in \eqref{warpassoc}, it follows from axiom \eqref{warpassoc} for the original skew warping. For \eqref{warpunit1}, we have the diagram
$$
\xymatrix{
T(TGK\star B) \ar[rr]^-{T(T\epsilon_K\star 1)} \ar[d]_-{v_{GK,B}} && T(TK\star B) \ar[rr]^-{T(v_0\star 1)} \ar[d]^-{v_{K,B}} && T(I\ox B) \ar[d]^-{T\lambda_B} \\
TGK\otimes TB \ar[rr]_-{T\epsilon_K\otimes 1} && TK\ox TB \ar[r]_-{v_0\ox 1} & I\otimes TB\ar[r]_-{\lambda_{TB}}& TB}
$$
which uses naturality of $v$ and axiom \eqref{warpunit1} for the original skew warping.
The next diagram proves \eqref{warpunit2}.
$$
\xymatrix{
TA\ar[r]^-{Ta} \ar[dd]_-{1} & TGA \ar[ldd]^-{T\epsilon_A} \ar[r]^-{TGk_A}& TG(TA\star K)\ar[r]^-{T\gamma^{-1}}
\ar[d]_-{T\epsilon_{TA\star K}} & T(TA\star GK)\ar[d]^-{v_{A,GK}}\ar[ld]^-{T(1\star \epsilon_K)} \\
& & T(TA\star K)\ar[rd]^-{v_{A,K}} & TA\ox TGK\ar[d]^-{1\ox T\epsilon_K} \\
TA \ar[rr]_-{\rho_{TA}}\ar[rru]^-{Tk_A} & & TA\ox I & TA\ox TK \ar[l]^-{1\ox v_0} 
}$$
Precomposing the next diagram with $1\star b \dd TA\star B \to TA\star GB$ proves \eqref{warpunit3}. 
Take note here of which components of $\gamma$ are required to be invertible.
$$
\xymatrix{
TA\star GB \ar[d]_-{1\star Gk_B}\ar[r]^-{\gamma} & G(TA\star B)\ar[r]^-{Gk_{TA\star B}}\ar@/_/@{->}[lddd]^-{G(1\star k_B)} & G(T(TA\star B)\star K)\ar[d]^-{\gamma^{-1}}\ar[ld]_-{1} \\
TA\star G(TB\star K)\ar[dd]_-{\gamma} & \ G(T(TA\star B)\star K)\ar[d]^-{G(v_{TA,B}\star 1)} & T(TA\star B)\star GK\ar[d]^-{v_{A,B}\star 1}\ar[l]_-{\gamma} \\
 & G((TA\ox TB)\star K)\ar[ld]_-{G\alpha} & (TA\ox TB)\star GK \ar[d]^-{\alpha}\ar[l]_-{\gamma} \\
G(TA\star (TB\star K))\ar[r]_-{\gamma^{-1}} & TA\star G(TB\star K)\ar[r]_-{1\star \gamma^{-1}} & TA\star (TB\star GK)
}$$
Then
$$
\xymatrix{
GK\ar[r]^-{\delta_K}\ar@/_/@{->}[rd]_-{1} & G^2K\ar[r]^-{Gk_{GK}}\ar[d]^-{G\epsilon_K} & G(TGK\star K)\ar[r]^-{\gamma^{-1}}\ar[d]^-{G(T\epsilon_K\star 1)} & TGK\star GK \ar[d]^-{T\epsilon_K\star 1} \\
& GK\ar[r]^-{Gk_K}\ar@/_/@{->}[rdd]_-{1} & G(TK\star K)\ar[r]^-{\gamma^{-1}}\ar[d]^-{G(v_0\star 1)} & TK \star GK\ar[d]^-{v_0\star 1} \\
& & G(I\star K)\ar[d]^-{G\lambda_K} & I\star GK\ar[d]^-{\lambda_{GK}}\ar[l]^-{\gamma} \\
& & GK\ar[r]_-{1} & GK
}$$
yields \eqref{warpunit4}, which completes the proof.
\end{proof}

\begin{corollary}\label{liftwarpcor1}
Under the hypotheses of Proposition~\ref{liftwarping}, the functor $\mathrm{U}\dd \CA^G\to \CA$
preserves the tensor products obtained from the skew warpings via Proposition~\ref{newskew}
and becomes opmonoidal when equipped with the unit constraint $\epsilon_I \dd GI\to I$.  
\end{corollary}

\begin{corollary}\label{liftwarpcor2}
Since $\CC$ is an object of $\Cat^\CC$ with its own tensor product as skew action,
and since it supports the identity skew warping, 
for any comonad $(\CC, G,\gamma, \delta, \epsilon)$ in the 2-category $\Cat^\CC$,
Corollary~\ref{liftwarpcor1} applies to give a skew monoidal structure on $\CC^G$
with $\mathrm{U}\dd \CC^G\to \CC$ opmonoidal. 
\end{corollary}

\begin{remark} If the comonad of Corollary~\ref{liftwarpcor2} is idempotent and $(G,\gamma)$
is strong in $\Cat^\CC$ then $U\dd \CC^G\to \CC$ is a coreflection and the dual of Theorem~\ref{skmonrefthm} applies. The same skew monoidal structure on $\CC^G$ is obtained as in 
Corollary~\ref{liftwarpcor2}.
The point is that the diagram \eqref{link} commutes by $G$ applied to the right-hand diagram
of \eqref{am3} and a counit property of the comonad.
So Theorem~\ref{skmonrefthm} appears to be a stronger result than 
Corollary~\ref{liftwarpcor2} in the idempotent comonad case.  
\begin{equation}\label{link}
\begin{aligned}
\xymatrix{
G(X\otimes GY) \ar[rr]^-{G(1\otimes \varepsilon_Y)} \ar[d]_-{G\gamma_{X,GY}} && G(X\otimes Y) \ar[d]^-{\delta_{X\otimes Y}} \\
GG(X\otimes Y) \ar[rr]_-{1} \ar[rru]_-{G\varepsilon_{X\otimes Y}} && GG(X\otimes Y)}
\end{aligned}
\end{equation}        
\end{remark}
   
\section{The example of Section~\ref{ex} without injectivity}

Let $\CC$ be a category with object set $O$ and morphism set $E$, and let $\xi\colon U\to O$ be a function (not necessarily injective).
Composition with $\xi$ induces a comonadic functor $N = \xi_!\colon\Set/U\to\Set/O$; write $R = \xi^*$ for the right adjoint, given by pullback. The comonad $G= NR = \xi_!\xi^*$ is given by $-\x_O U$.

The category structure on $\CC$ induces a skew monoidal structure on $\Set/O$, with tensor product $X\ox Y$ given by:
$$(X\ox Y)_j = \sum_i X_i \x \CC(i,j)\x Y_j$$
and so $X\ox-$ is given by $X\x_O E\x_O -$. 
The unit $I$ is the terminal object $1\colon O\to O$. 

From the formulas for $G$ and $X\ox-$ involving products in $\Set/O$, it is clear that we have natural isomorphisms $\gamma_{X,Y}\colon X\ox GY\cong G(X\ox Y)$, compatible with the comonad structure, in the sense that the diagrams \eqref{am3} commute.
Almost as easy is compatibility with the associativity map and left unit constraint in the sense of diagrams \eqref{am1} and \eqref{am2}.

So we have a category $\CC$ with object-set $O$, giving rise to the skew monoidal category $\Set/O$, and the comonad $G=\xi_!\xi^*$ on $\Set/O$ as required by Corollary~\ref{liftwarpcor2}. This gives rise to a skew monoidal structure on $\Set/U$, with unit $\xi^*I$; in other words with unit $I'$ equal to the terminal object $1\colon U\to U$. It is clear from the construction that this tensor product preserves colimits in each variable. So from the general theory, it must correspond to some category $\CA$ with object-set $U$. 

Since $\xi^*\colon\Set/U\to\Set/O$ is opmonoidal, $\xi$ is the object part of a functor $F\colon\CA\to\CC$. Since $\xi^*$ preserves the tensor, the functor $F$ is fully faithful. y

Thus $\CA$ must in fact be obtained from $\xi \colon U\to\CC$ via the factorization into
a bijective-on-objects functor followed by a fully faithful functor.

\end{document}